\documentclass[leqno]{amsart} 

\usepackage[section]{placeins}

\usepackage[margin=1.5in]{geometry} 
\usepackage[utf8]{inputenc} 
\usepackage{amsmath} 
\usepackage{amsthm} 
\usepackage[upint]{stix2} 
\usepackage{inconsolata} 
\usepackage[shortlabels]{enumitem} 
\setlist[enumerate]{
     itemsep  = 0.2cm,
       label  = {\upshape (\arabic*)},
         ref  = \arabic*,
    leftmargin  = *
} 
\usepackage{xcolor} 
\usepackage{tikz-cd} 
\tikzcdset{arrow style=tikz, diagrams={>={Computer Modern Rightarrow[width=5pt,length=3pt]}}} 
\usepackage[colorlinks,allcolors=blue]{hyperref} 
\usepackage[nameinlink]{cleveref} 
\usepackage[final]{microtype}
\usepackage{todonotes} 

\usepackage{quiver}

\usepackage{comment}

\usepackage[most]{tcolorbox}

\usepackage{CJKutf8} 

\numberwithin{equation}{section} 
\newtheorem{corollary}[equation]{Corollary} 
\newtheorem{lemma}[equation]{Lemma}
\newtheorem{question}[equation]{Question}

\newtheorem{theorem}[equation]{Theorem}

\theoremstyle{definition} 

\newtheorem{remark}[equation]{Remark}
\newtheorem{conjecture}{Conjecture}

\newcommand{\subgrp}{\leq}

\newcommand{\cp}{C_p\rtimes C_{(p-1)/2}}

\newcommand{\PSL}{\operatorname{PSL}}
\newcommand{\Sz}{\operatorname{Sz}}
\newcommand{\Sol}{\operatorname{Sol}}
\newcommand{\Solv}{\operatorname{Solv}}

\usepackage[style=ieee,citestyle=numeric,backend=biber,sorting=nyt,url=false]{biblatex}

\AtEveryBibitem{\clearfield{issn}}	
\AtEveryCitekey{\clearfield{issn}}
\AtBeginBibliography{\small}		

\DeclareFieldFormat{postnote}{#1}

\DeclareFieldFormat[article,inbook,incollection,inproceedings,patent,thesis,unpublished]{title}{\textit{#1\isdot}} 

\DeclareFieldFormat{journaltitle}{#1}  

\DeclareFieldFormat[book,inbook,incollection,inproceedings]{series}{#1} 

\defbibheading{references}[\refname]{%
\section*{#1}%
\addcontentsline{toc}{section}{References} %
\markboth{#1}{#1}
}

\title{The Number of Solvabilizers in Finite Groups}

\author[B.~Akbari]{Banafsheh~Akbari}
\address{Department of Mathematics\\
Tufts University\\
Medford, MA 02155,
USA}
\email{akbari@banafsheh.net}

\author[E.~Han]{Ethan~Han}
\address{Department of Mathematics \\ University of California, Berkeley \\ Berkeley, CA 94720, USA}
\email{ethan\_han@berkeley.edu}

\author[S.~Lin]{Sasha~Lin}
\address{Department of Mathematics \\University of California, Irvine\\ Irvine, CA 92697, USA}
\email{linaa1@uci.edu}

\author[B.~Vakil]{Benjamin~Vakil}
\address{Department of Mathematics \\ University of Connecticut \\ Storrs, CT 06269, USA}
\email{benjamin.vakil@uconn.edu}

\date{}

\begin{document}

\begin{abstract}
Considering a finite group $G$, for any element $x\in G$, the solvabilizer of $x$ in $G$ is defined as $\Sol_G(x)=\{y \in G : \langle x, y \rangle \text{ is solvable}\}$. In this paper, we introduce $\Solv(G)$ as the number of distinct solvabilizers of elements in $G$. A group is called $n$-solvabilizer if $|\Solv(G)|=n$. 
We compute $|\Solv(G)|$ for various classes of non-abelian simple groups, including $\PSL(2, 2^n)$; $\PSL(2, 3^n)$ with an odd integer $n$; and $\PSL(2, p)$ with a prime $p>7$. Furthermore, we show that for any nonsolvable group $G$, $|\Solv(G)|\geq 32$. Finally, we implement an algorithm in GAP for calculating $|\Solv(G)|$ for any nonsolvable group $G$. This algorithm can be adapted for all questions generalizing to nilpotent and other subgroup-closed classes of finite groups.
\end{abstract}

\maketitle

\section{Introduction
}
For any element $x$ in a finite group $G$, we define the solvabilizer of $x$ in $G$ as 
$$\Sol_G(x)=\{y \in G : \langle x, y \rangle \text{ is solvable}\}.$$

This set has been studied extensively in \cite{AkbariLewis, tables, Carmine}.
The solvabilizers in a group play analogous roles to the centralizers where the centralizer of an element $x\in G$, $C_{G}(x)$, can be thought of as the elements $y\in G$ where $\langle x,y\rangle$ is \textit{abelian}. This brings up the problem of computing the number of distinct solvabilizers of elements which is analogous to a previously studied problem on determining the number of distinct centralizers in a group \cite{Belcastro}. This can be considered as a new exploration of the combinatorial properties of groups. 

Given group $G$, we denote $$\Solv(G)=\{\Sol_G(x) : x\in G\}.$$ The group $G$ is called $n$-solvablizer if $|\Solv(G)|=n$. Due to a known result by Thompson \cite{Thompson}, a group $G$ is solvable if and only if for any two elements $x$ and $y$ in $G$, the subgroup $\langle x, y \rangle$ is solvable. As a result,  $G$ is $1$-solvablizer if
and only if $G$ is solvable. For the alternating group $A_5$ as the smallest nonsolvable group, we find that $|\Solv(A_5)|=32$. We will prove that for any nonsolvable group $G$, $|\Solv(G)|\geq 32$. 
\\ \\
In Section \ref{minimalsimplegroups}, we compute $|\Solv(G)|$ for any minimal simple group $G$, which is a non-abelian simple group whose proper subgroups are solvable. They are completely classified by Thompson \cite{Thompson}. Using the list of numbers we collect as the number of solvabilizers for the minimal simple groups, we will show that any two minimal simple groups have different number of solvabilizers. This raises the following question.
\begin{question}
Let $G$ and $H$ be two $n$-solvablizer finite simple groups. Is $G$ isomorphic to $H$?
\end{question}
Those numbers giving the positive answer to this question provide a new way to characterize all finite groups that are $n$-solvablizer. 

In Section \ref{othersimplegroups}, we generalize the results obtained in \cite{tables} presenting the structure of solvabilizers of elements in the minimal simple groups and provide the whole structure of solvabilizers in the projective special linear group of type $\PSL(2, 2^n)$; $\PSL(2, 3^n)$ for any odd integer $n$; and $\PSL(2, p)$ for any prime $p$. As a result we show that our formulas for $|\Solv(G)|$ found for the minimal simple groups $G$ in Section \ref{minimalsimplegroups} can be generalized for all these projective special linear groups.
\\ \\
In Section \ref{nonsolvablesection}, we discuss the lower bounds on $|\Solv(G)|$ for all nonsolvable groups $G$. We also conjecture that any $32$-solvablizer group has a composition factor isomorphic to $A_5$.
\\ \\
Finally, in Section \ref{GAPCalculation}, we implement an algorithm to efficiently calculate $|\Solv(G)|$ for any group $G$. The improvement over the naive algorithm of iterating over every element is achieved by instead iterating over rational classes of elements in $G$ where the rational class of an element $x$ indicates the union of all the conjugacy classes of $x^{i}$ for an element $x\in G$ where $\gcd(|x|,i)=1$.  Computations were performed in GAP, and the code can be found at the link in that section.

\section{Preliminary results}

We will demonstrate several lemmas concerning the number of solvablizers in groups. As we mentioned above, a group $G$ is solvable if and only if, for any two elements $x, y\in G$  the subgroup
$\langle x, y\rangle$ is solvable. This means that a group $G$ is solvable if and only if $|\Solv(G)|=1$. It is natural to ask if there exists a nonsolvable group $G$ with $|\Solv(G)|=2$. Generally, the following question arises.

\begin{question}
What natural numbers can occur as $|\Solv(G)|$?
\end{question}
To that end, we can state the following lemma. Which we prove later:
\begin{lemma}
    The number of distinct solvabilizers in $A_5$ is $32$.
\end{lemma}
Moreover, any group of order less than or equal to 59 is $1$-solvabilizer. However, the question of whether there exists any nonsolvable group $G$ with $|\Solv(G)|\leq 31$, still remains. To answer this question, we apply the following theorem in \cite{BarryWard}.
\begin{theorem}{\rm \cite[Theorem 1]{BarryWard}}\label{groupcontainingminimalsimplegroups}
    Any non-abelian simple group $G$ contains a subgroup which is a minimal simple group.
\end{theorem}
The following result along with Theorem \ref{groupcontainingminimalsimplegroups} asserts that the number of distinct solvabilizers in minimal simple groups are the smallest ones among all non-abelian simple groups. 
\begin{lemma}\label{sovH<SolvG}
    Let $G$ be a group and $H$ a subgroup of $G$. Then we have $|\Solv(H)| \leq |\Solv(G)|$.
\end{lemma}

\begin{proof}
    Suppose $x, y \in H$ with $\Sol_H(x) \neq \Sol_H(y)$.  Then, without loss of generality, there exists $z \in H$ such that $\langle x, z \rangle$ is solvable but $\langle y, z\rangle$ is not solvable. Considering the fact that $\Sol_H(x)=\Sol_G(x)\cap H$, we can see that $z\in \Sol_G(x)$ and $z\notin \Sol_G(y)$ which implies that $\Sol_G(x) \neq \Sol_G(y)$. If we had that $|\Solv(H)|>|\Solv(G)|$, we would necessarily have $x,y\in H$ such that $\Sol_{H}(x)\neq\Sol_{H}(y)$ but $\Sol_{G}(x)=\Sol_{G}(y)$, which is impossible.
\end{proof}

Theorem \ref{groupcontainingminimalsimplegroups} and Lemma \ref{sovH<SolvG} draw our attention to focus primarily on finding $|\Solv(G)|$ for the minimal simple groups. Once
we compute it in Section \ref{minimalsimplegroups}, we will notice that $|\Solv(A_5)|=32$ is the smallest number appearing in our list as the number of distinct solvabilizers in a minimal simple group. So after applying Theorem \ref{groupcontainingminimalsimplegroups} and Lemma \ref{sovH<SolvG}, it is seen that $|\Solv(G)|\geq 32$ for any non-abelian simple group $G$. We will prove further that $|\Solv(G)|\geq 32$ for any nonsolvable group $G$.
\\ \\
In the sequel, we present a result enabling us to reduce the problem on finding $|\Solv(G)|$ to the Fitting-free groups, examples of which include direct and wreath products of almost simple groups. To prove it, we apply the following lemma in \cite{AkbariLewis}. It is worth mentioning that $\frac{\Sol_{G}(x)}{R(G)}$ indicates the set $\{yR(G) \mid  y \in \Sol_{G}(x)\}$, where 
$R(G)$ is the solvable radical of $G$.
\begin{lemma}{\rm (\cite[Lemma 2.5]{AkbariLewis})}\label{SolG/R(G)}
    Let $G$ be a nonsolvable group and $x$ be an element in $G$. Then $\Sol_{G/R(G)}(xR(G))=\frac{\Sol_{G}(x)}{R(G)}$.
\end{lemma}
\begin{lemma}\label{SolvG/R(G)}
 Let $G$ be finite group with $R(G)\neq 1$. Then   $|\Solv(G)| = |\Solv(G/R(G))|$.
\end{lemma}
\begin{proof} Consider the following map from $\Solv(G/R(G))$ to $\Solv(G)$:
$$\Sol_{G/R(G)}(xR(G)) \mapsto \Sol_G(x)$$

We must show that it is a bijection. Considering Lemma \ref{SolG/R(G)}, it is evident that the map is well-defined. In what follows, we prove that it is injective.  Suppose $\Sol_{G/R(G)}(xR(G)) \neq \Sol_{G/R(G)}(yR(G))$.  Then without loss of generality, for some $zR(G)$, $\langle xR(G), zR(G)\rangle$ is solvable but $\langle yR(G), zR(G)\rangle$ is not. 
Then it can be easily seen that
\begin{equation*}
    \langle xR(G), zR(G)\rangle\cong \frac{\langle x, z \rangle}{\langle y, z \rangle \cap R(G)}.
\end{equation*}
This implies that $\langle x, z\rangle$ is solvable. By a similar argument, $\langle y, z \rangle$ is not solvable. Thus $\Sol_{G}(x) \neq \Sol_{G}(y)$ and so this map is injective.  It is clear to see that it is also surjective which completes the proof.
\end{proof}


\section{The number of solvabilizers in minimal simple groups}\label{minimalsimplegroups}
In this section, we provide a complete classification for the size of $\Solv(G)$ where $G$ is a minimal simple group.
\\ \\
Thompson \cite[Corollary 1]{Thompson} classified all minimal simple groups. Based on his classification, every minimal simple group is isomorphic to one of the following groups: $\PSL(2, 2^p)$ where $p$ is any prime; 
$\PSL(2, 3^p)$ where $p$ is an odd prime; $\PSL(2, p)$ where $p > 3$ is a prime satisfying $p \equiv 2,3 \pmod 5$; ${\rm Sz}(2^p)$ where $p$ is an odd prime; and $\PSL(3, 3)$.
\\ \\
It was stated in \cite{AkbariLewis} that the solvabilizer of an element $x$ in $G$ is $$\Sol_{G}(x)=\bigcup_{x\in H\subgrp G} H,$$ where the union ranges over all solvable subgroups $M$ of $G$ containing $x$. It is not difficult to see that this definition can be improved as the union ranges over all maximal solvable subgroups of $G$ containing $x$.
Given a minimal simple group $G$, Since all proper subgroups of $G$ are solvable, the solvabizer of an element $x\in G$ can be simply taken as the union ranges over all maximal subgroups of $G$ containing $x$.  
\\ \\
The paper \cites{tables} characterizes the solvabilizers of elements together with their sizes in all minimal simple groups where they are represented as the union of some maximal subgroups up to isomorphism (see also Tables \ref{tbl2,2p}--\ref{tbl2,p,23} in the Appendix).  However, it cannot exactly describe how many elements in $G$
produce the same solvabilizers.  So we need to delve in the structure of solvabilizers to find such elements.
\\ \\
Considering the definition above for the solvabilizers of a minimal simple group $G$, we show that given two elements 
$x$ and $y$, if $\Sol_{G}(x)=\Sol_{G}(y)$, then they are exactly contained in the same maximal subgroups of $G$. This also raises an interesting question that we do not fully address in this paper:
\begin{question}\label{Sol(x)=Sol(y) implies same subgroups}
    For what groups $G$ does $\Sol_{G}(x)=\Sol_{G}(y)$ imply that $x$ and $y$ are contained in the same maximal solvable subgroups? 
\end{question}

Clearly, this question has a particular application to calculating $|\Solv(G)|$.
\\ \\
The proof of this for minimal simple groups requires us to apply the results for the  solvabilizers in minimal simple groups found in \cite{tables}. For convenience, we include them in some tables in the Appendix.
\\ \\
While proving the next theorems, we will observe that any involution, an element of order $2$, in a minimal simple group $G$ produces a unique solvabilizer, so, $|\Solv(G)|\geq |I(G)|$ where $I(G)$ denotes the set of all involutions. 
\begin{theorem}
    Let $G$ be a minimal simple group other that $\PSL(3, 3)$.  If the elements $x$ and $y$ in $G$ satisfy $\Sol_G(x) = \Sol_G(y)$, then $x$ and $y$ are contained in exactly the same maximal subgroups.
\end{theorem}
    To prove the theorem above, we proceed case by case for the four classes of minimal simple groups.
    \begin{lemma}\label{maximalpsl(2,2^p)}
    Let $x$ and $y$ be two elements in $G=\PSL(2, q)$ with $q=2^p$ and $p$ as an odd prime, satisfying $\Sol_G(x) = \Sol_G(y)$. Then $x$ and $y$ are contained in exactly the same maximal subgroups.
    \end{lemma}
    \begin{proof}
    We consider Table \ref{tbl2,2p} in the Appendix where it lists the number of maximal subgroups containing a certain element in $G$ (see also \cites[Theorem 3.1]{tables}). Assuming $Sol_G(x) = Sol_G(y)$, it is seen that the non-identity elements $x$ and $y$ are both involutions, both of order dividing $q-1$, or both of order dividing $q+1$. In the case where $|x|$ and $|y|$ divide $q+1$, the solvabilizer is a single copy of $D_{2(q+1)}$. Therefore, they are obviously contained in the same maximal subgroup. In the case where $|x|$ and $|y|$ divide $q-1$, we let $x$ be contained in $K_1$ and $K_2$ isomorphic to $C_2^p \rtimes C_{q-1}$, and $H_1$ isomorphic to $D_{2(q-1)}$. Additionally, $y$ is contained in $K_3$ and $K_4$ isomorphic to $C_2^p \rtimes C_{q-1}$ and $H_2$ isomorphic to $D_{2(q-1)}$. Then $$\Sol_G(x)=K_1\cup K_2\cup H_1=K_3\cup K_4\cup H_2=\Sol_G(y).$$ 
    Assume first that $K_i$'s are distinct. Clearly, the intersection of any two of them can not contain an involution. So all involutions in $K_1\cup K_2$ 
must be covered by $H_2$ where there are $2(q-1)$ involutions in $K_1\cup K_2$. This is impossible because  $H_2$ only contains $q-1$ involution. So without loss of generality we assume $K_2 = K_3$. 
\\ \\
Next, suppose $K_1 \neq K_4$. In what proceeds, we will show that $K_2\cup K_4\cup H_2$ cannot cover all nonidentity elements of $K_1$ of order dividing $q-1$, whose number is $q(q-1)-q = q(q - 2)$. Note that $K_1 \cap K_2$ and $K_1 \cap K_4$ are isomorphic to a subgroup of $C_{q-1}$, and $H_2 \cong D_{2(q-1)}$ has 
$q - 2$ nonidentity elements of order dividing $q-1$. Thus $K_2\cup K_4\cup H_2$ can only cover at most $3(q - 2)$ elements of order dividing $q-1$ in $K_1$, since the intersection of any of these three subgroups with $K_1$ is at most a copy of $C_{q-1}$.  This is less than $q(q - 2)$ for all $q=2^p$, a contradiction. Therefore $K_1 = K_4$ and so $x$ and $y$ are contained in the same copies of $C_2^p \rtimes C_{q-1}$. It follows that $x$ and $y$ must be contained in the same copy of $C_{q-1} \cong K_1\cap K_2$.  So they are contained in exactly the same maximal subgroups.
    \\ \\
    Lastly, consider when $x$ and $y$ are involutions. 
    From  Table \ref{tbl2,2p}, $x$ belongs to one copy of $C_2^p \rtimes C_{q-1}$, $q/2$ copies of $D_{2(q+1)}$, and $q/2$ copies of $D_{2(q-1)}$.
    Let $H$ be a copy of $D_{2(q+1)}$ containing $x$. Take an element $z \in H$ with order dividing $q+1$. Since $z\in \Sol_G(x)$, we have $z\in \Sol_G(y)$ which implies that $z$ is in some copy of $D_{2(q+1)}$ containing $y$. As $H$ is a unique copy of $D_{2(q+1)}$ containing $z$,  we have $y\in H$.
    This works for any of the $q/2$ copies of $D_{2(q+1)}$ containing $x$. Hence each group of this type also contains $y$. However, since the intersection of any two $D_{2(q+1)}$'s contains at most one involution, we have $x=y$.  This also certifies that there are $q^2-1$ solvabilizers of involutions, since there are $q^2 - 1$ involutions each with a distinct solvabilizer..
    \end{proof}
    \begin{lemma}\label{maximalpsl(2,3^p)}
    Let $x$ and $y$ be two elements in $G=\PSL(2, q)$ whit $q=3^p$ and $p$ as an odd prime, satisfying $\Sol_G(x) = \Sol_G(y)$. Then $x$ and $y$ are contained in exactly the same maximal subgroups.
    \end{lemma}
    \begin{proof}
    According to the size of the solvabilizers of elements in $G$ collected in Table \ref{tbl2,3p} in the Appendix, the nonidentity elements $x$ and $y$ are both involutions, both order $3$, both of order dividing $q-1$ or both of order dividing $q+1$ (see also \cites[Theorem 3.2]{tables}).
    \\ \\
    Let us first suppose that $x$ and $y$ are involutions. It is seen from Table \ref{tbl2,3p} that there are exactly $\frac{q+3}{2}$ copies of $D_{q+1}$ containing $x$. Given a copy of $D_{q+1}$ containing $x$, we take some non-involution $z \in D_{q+1}$ with order dividing $q+1$. Note that $z\in \Sol_G(x)$ and assuming $\Sol_G(x) = \Sol_G(y)$, we find that $z$ is contained in a copy of $D_{q+1}$ containing $y$. However, there is only one copy of $D_{q+1}$ containing $z$. Therefore, $x$ and $y$ must be in this copy of $D_{q+1}$. This holds for every copy of $D_{q+1}$ containing $x$, so $x$ and $y$ are contained in the intersection of these $D_{q+1}$'s, which is a copy of $C_2$. Therefore $x=y$.
    \\ \\
    Now let $x$ and $y$ be elements of order dividing $q+1$. Considering Table $2$, there is a unique copy of $D_{q+1}$ where $\Sol_G(x) = \Sol_G(y)=D_{q+1}$. Consequently, we also have that $x$ and $y$ are contained in the same maximal subgroup.
    \\  \\
    We now suppose that $x$ and $y$ are elements of order dividing $q-1$. By a similar argument to Lemma \ref{maximalpsl(2,2^p)}, we can show that $x$ and $y$ are in the same two copies of $C_3^p \rtimes C_{(q-1)/2}$ and one copy of $D_{q-1}$.
    \\ \\
    Finally, we consider when $x$ and $y$ have order $3$.  We wish to show that $y=x$ or $y=x^2$. For sake of contradiction, suppose that this is not true, so $\langle x\rangle \neq \langle y \rangle$.  We find from the table that $x$ (respectively $y$) is contained in one copy of $C_3^p \rtimes C_{(q-1)/2}$ and $q/3$ copies of $A_4$. Note that the number of elements of order dividing $q-1$ in a copy of $C_3^p \rtimes C_{(q-1)/2}$ is
    $\frac{q(q-1)}{2} - q = \frac{q(q-3)}{2}$. On the other hand, the intersection of any two subgroups of type $C_3^p \rtimes C_{(q-1)/2}$ is at most a copy of  $C_{(q-1)/2}$, because any element of order $3$ is contained in exactly one 
    $C_3^p \rtimes C_{(q-1)/2}$. Therefore, the two copies of $C_3^p \rtimes C_{(q-1)/2}$ must be the same, in order for the solvabilizers to contain the same elements of order dividing $q-1$.
    \\ \\
    We now show that the copies of $A_4$ containing $x$ and $y$ are the same.
    Note that all copies of $C_3$ are conjugate in $G$,
    so there exists some $g \in G$ such that $g\langle x\rangle g^{-1} = \langle y \rangle$.  We can assume that $gxg^{-1} = y$ since otherwise, $gxg^{-1} = y^2$ and thus we can replace $y$ with $y^2$ to give the same argument. 
    \\ \\
    As we demonstrated above, there is only one copy of $C_3^p \rtimes C_{(q-1)/2}$ containing $x$. We take it as $K$. Then we have $$K^g\subseteq \Sol_G(x)^g=\Sol_G(x^g)=\Sol_G(y).$$ Since $K$ is the only copy of  $C_3^p \rtimes C_{(q-1)/2}$ containing $x$ and $y$, $K = gKg^{-1}$. This implies that $g \in N_G(K) = K$. On the one hand, $K$ contains a unique abelian Sylow $3$-subgroup. we see that $g$ cannot be contained in this Sylow $3$-subgroup since otherwise,  $y=gxg^{-1}= xgg^{-1}$ which is impossible. Therefore, $g$ must be an element of $K$ with order dividing $\frac{q-1}{2}$.
    \\ \\
    According to \cites[Theorem 2.1]{BlythRobinson}, there are $\frac{q^2-1}{2}$ copies of $C_3$ where they are all conjugate. Thus given a copy of $C_3$, we get  $$|G:N_G(C_3)|=\frac{q(q^2-1)/2}{|N_G(C_3)|}=\frac{q^2-1}{2},$$ which yields $|N_G(C_3)|=q$. This means that $N_G(C_3)$ is a unique Sylow $3$-subgroup of order $q$ contained in the unique copy of $C_3^p \rtimes C_{(q-1)/2}$ containing $C_3$.
    \\ \\    
    Let $H_1, H_2, \ldots, H_{q/3}$ be the copies of $A_4$ containing $x$.  Each $H_i$ contains four subgroups isomorphic to $C_3$. One of these subgroups is $\langle x\rangle$,  which is contained in $K$, while the other three are not, as the intersection of $K$ and $H_i$ is exactly $\langle x \rangle$. On the other hand, given two distinct copies $H_i$ and $H_j$ of $A_4$, since $H_i \cap H_j = \langle x \rangle$, all subgroups $H_1, H_2, \ldots, H_{q/3}$ have 
    $3 \cdot (q/3) = q$ distinct copies of $C_3$ not equal to $\langle x \rangle$ not contained in $K$. Denote these copies of $C_3$ by $J_1, J_2, \ldots , J_q$.
    \\ \\    
    After conjugating by $g$, $J_1, J_2, \ldots, J_q$ are mapped to $gJ_1g^{-1}, gJ_2g^{-1}, \ldots, gJ_qg^{-1}$ where they are distinct from $\langle y\rangle$ not being contained in $K$. In addition, $\langle x\rangle$ is mapped to $\langle y \rangle$. Note that in order for the solvabilizers of $x$ and $y$ to be the same, they must contain the same copies of $C_3$. Therefore the two lists $J_1, J_2, \ldots, J_q$ and $gJ_1g^{-1}, gJ_2g^{-1}, \ldots, gJ_qg^{-1}$ are the same. We find the conjugation by $g$ \textit{permutes} $J_1, J_2, \ldots, J_q$, and denote this permutation by $\sigma_g$ where for any $1\leq i\leq q$, 
    $\sigma_g(J_i) = J_j$ for some $1\leq j\leq q$.
    \\ \\
    In what follows, we will show that $|\sigma_g|=|g|$. Let $n$ be a natural number such that ${\sigma_g}^n(J_i)=J_i$ for all $J_i$, $1\leq i\leq q$. Then by the definition of ${\sigma_g}$, we have $g^n J_i g^{-n} = J_i$, and so $g^n \in N_G(J_i)$. On the other hand, as shown before, $N_G(J_i)$ is a $3$-group since $J_i \cong C_3$. Therefore, $g^n$ has order divisible by $3$. We get $g^n=1$ since 
    $|g|\mid \frac{q-1}{2}$. This implies that $|g|\mid n$. Moreover, ${\sigma_g}^{|g|}$ is the identity permutation. Thus,  $|\sigma_g|=|g|$.
    However, this means that the cycle decomposition of $\sigma_g$ acting on $J_1, J_2, \ldots, J_q$ consists entirely of $|g|$-cycles. Considering $|g| \mid \frac{q-1}{2}$, it is coprime with $q$, a contradiction.  So we must have $\langle x\rangle = \langle y \rangle$, and we are done.
    \end{proof}
    \begin{lemma}\label{maximalSZ(2,2^p)}
    Let $x$ and $y$ be two elements in $G=\Sz(2, q)$ with $q=2^p$ and $p$ as an odd prime, satisfying $\Sol_G(x) = \Sol_G(y)$. Then $x$ and $y$ are contained in exactly the same maximal subgroups.
    \end{lemma}
    \begin{proof}
    To do this, we use Table \ref{tblsz,2p} in the Appendix where it lists all maximal subgroups containing an element of a prime order dividing $|G|$. In the table we set $q_{\pm}:= q \pm \sqrt{2q} + 1$ (see also \cite[Theorem 3.5]{tables}). 
    \\ \\
    First, assume that $x$ and $y$ are involutions. There are $q^2/4$ copies of $C_{q_+}\rtimes C_4$ containing $x$ (respectively $y$). Then an element $z$ of order $q_+$ in a copy of $C_{q_+}\rtimes C_4$, $\langle x, z\rangle$ (respectively $\langle y, z\rangle$) is solvable which implies that $x\in \Sol_G(z)$ (respectively $y\in \Sol_G(z)$). On the other hand, $z$ is contained in a unique copy of $C_{q_+} \rtimes C_4$. It is deduced from $\Sol_G(x)=\Sol_G(y)$ that $x$ and $y$ are contained in exactly the same copies of $C_{q_+} \rtimes C_4$. From the table, we see that the intersection of any two copies of $C_{q_+} \rtimes C_4$ can be a $C_2$ or $C_4$, so we conclude that $x = y$.
    \\ \\
    Now consider when $x$ and $y$ are both of order $4$. By the same logic used for the involutions, it is observed that they are contained in the same copies of $C_{q_+} \rtimes C_4$. On the other hand, the intersection of any two copies of  $C_{q_+} \rtimes C_4$ is a $C_4$. This implies that $x$ and $y$ are both contained in the same copy of $C_4$ and therefore in the same maximal subgroups.
    \\ \\
    Next, assume the case where $x$ and $y$ are of order dividing $q-1$. We  use a similar argument to Lemma \ref{maximalpsl(2,2^p)} where the two elements with order dividing $q-1$ and show that $x$ and $y$ are contained in the same copies of $(C_2^p .C_2^p) \rtimes C_{q-1}$. It follows that they are contained in the same copy of $C_{q-1}$ and thus in the same maximal subgroups.
    \\ \\ 
    Finally, suppose that $x$ and $y$ are both of order dividing $q_+$ or both of order dividing $q_-$. As the solvabilizers of $x$ and $y$ both are exactly one maximal subgroup, they must be contained in the same maximal subgroup.
    \end{proof}

    \begin{lemma}\label{maximalPSL(2,p)}
     Let $x$ and $y$ be two elements $G=\PSL(2, p)$ with $p > 5$ and $p \equiv 2,3 \pmod 5$, satisfying $\Sol_G(x) = \Sol_G(y)$. Then $x$ and $y$ are contained in exactly the same maximal subgroups.
    \end{lemma}
    \begin{proof}
    We demonstrate that any two involutions have distinct solvabilizers. To do this, we suppose that $x$ and $y$ are involutions with the same solvabilizers and try to show that $x=y$. According to the Tables \ref{tbl2,p,1}--\ref{tbl2,p,23} in the Appendix, $x$ and $y$ are contained in some copies of $D_{p+1}$. It is also observed that every element of order dividing $p+1$ is contained in exactly one copy of $D_{p+1}$. This yields that $x$ and $y$ are exactly in the same copies of $D_{p+1}$.
\\ \\
We will consider the cases where $p \equiv 1 \pmod 4$
and $p \equiv 3 \pmod 4$ separately. Assume first that $p \equiv 1 \pmod 4$. In this case, $C_{p} \rtimes C_{(p-1)/2}$ contains some involutions. So it is seen from Tables 4--11 that $x$ and $y$ as two involutions are each contained in two copies of $C_{p} \rtimes C_{(p-1)/2}$.  We also see that every element of order $p$ is contained in exactly one copy of $C_p \rtimes C_{(p-1)/2}$. This implies that $x$ and $y$ are contained in exactly the same copies of $C_p \rtimes C_{(p-1)/2}$. So they are contained in the intersection of a copy of $C_p \rtimes C_{(p-1)/2}$ and a copy of $D_{p+1}$ which has order at most $2$.  This forces $x$ to be equal to $y$. Consider next the case where $p \equiv 3 \pmod 4$.  Then there is no copy of $C_p \rtimes C_{(p-1)/2}$ containing $x$ and $y$. We also see from the Tables \ref{tbl2,p,1}-\ref{tbl2,p,23} that there are some copies of $D_{p-1}$ containing $x$ and $y$. On the other hand, it is observed by the tables that 
there is only one copy of $D_{p-1}$ containing any element of order dividing $p-1$. It follows that $x$ and $y$ are contained in exactly the same copies of $D_{p-1}$. So they are contained in the intersection of a copy of $D_{p-1}$ and a copy of $D_{p+1}$ which implies that $x = y$. 
\\ \\
We assume now that $x$ and $y$ are both of order $p$. According to the tables, $x$ and $y$ are each contained in exactly one copy of $C_p \rtimes C_{(p-1)/2}$. So considering $\Sol_G(x)=\Sol_G(y)$, we can find that $x$ and $y$ are contained in the same copy of $C_p \rtimes C_{(p-1)/2}$.  Noting that $C_p \rtimes C_{(p-1)/2}$ has a unique Sylow $p$-subgroup, it is seen that they are contained in the same copy of $C_p$.
\\ \\
If  $x$ and $y$ are both of order dividing $p + 1$, it is seen again from Tables \ref{tbl2,p,1}--\ref{tbl2,p,23} that $x$ and $y$ are each contained in exactly one copy of $D_{p+1}$ which follows that 
they are contained in the same copy of $D_{p+1}$ having a unique subgroup isomorphic to $C_{(p-1)/2}$. Therefore,  $x$ and $y$ are both in this copy of $C_{(p-1)/2}$.
\\ \\
Suppose that $x$ and $y$ are both of order dividing $p - 1$. We see from Tables \ref{tbl2,p,1}--\ref{tbl2,p,23} that every element of order dividing $p-1$ is contained in two copies of  $C_p \rtimes C_{(p-1)/2}$ 
and one copy of $D_{p - 1}$. Since the solvabilizers of $x$ and $y$ are equal each containing two copies of  $C_p \rtimes C_{(p-1)/2}$, they must cover the same elements of order $p$. This implies that $x$ and $y$ are contained in the same copies of $C_p \rtimes C_{(p-1)/2}$ which we call $H_1$ and $H_2$. As $x$ is contained in a unique copy of $D_{p-1}$ and so its unique subgroup $C_{(p-1)/2}$, we can express $x = z^k$ for some $k$ where $\langle z\rangle\cong C_{(p-1)/2}$ . Note that $z$ has order $\frac{p-1}{2}$ and is contained in $H_1$ and $H_2$ since otherwise there is one more copy of  $C_p \rtimes C_{(p-1)/2}$ containing $x$ which is impossible. By a similar argument, we can see that $z$ is contained in the same copy of $D_{p-1}$ as $x$. Therefore, $\Sol_G(z) = \Sol_G(x)$. On the other hand, $C_{(p-1)/2}\cong\langle z \rangle \subseteq H_1 \cap H_2$ and $C_{(p-1)/2}$ is a  maximal subgroup of   $C_p \rtimes C_{(p-1)/2}$ which forces that  $H_1 \cap H_2 = \langle z \rangle$. However, $y$ is also contained in $H_1 \cap H_2$. So $x$ and $y$ are both powers of $z$ which means that they belongs to the same cyclic subgroup. Thus, they are contained in the same maximal subgroups. 
\\ \\
Next, let $x$ and $y$ be two elements of order $3$. We consider the two possible cases $p\equiv 1\pmod 6$ and $p\equiv 5\pmod 6$ separately. If $p\equiv 1\pmod 6$ , then it is seen from the tables that $x$ and $y$ are contained in two copies of $C_p \rtimes C_{(p-1)/2}$, one copy of $D_{p-1}$ and some copies of $A_4$ or $S_4$. By the same reasoning as the case where $x$ and $y$ have order dividing $p-1$, we can conclude that $x$ and $y$ are contained in the same copies of $C_p \rtimes C_{(p - 1)/2}$. We can also find an element $z$ of order $\frac{p-1}{2}$ such that $x \in \langle y \rangle$ and so $z$ is in $H_1$ and $H_2$.  By a similar method to the previous case, we can show that $H_1 \cap H_2 = \langle z \rangle$ and so $y \in \langle z \rangle$.  On the other hand, $\langle z \rangle$ as a cyclic of order $\frac{p-1}{2}$ has exactly two elements of order $3$ which yields that either $x = y$ or $y=x^2$.  Therefore, $x$ and $y$ are in the same maximal subgroups. If $p\equiv 5\pmod 6$, then $x$ and $y$ are not in any copies of $C_p \rtimes C_{(p - 1)/2}$. In this case, $3\mid p+1$. We can now observe in Tables \ref{tbl2,p,1}-\ref{tbl2,p,23} that any element of order $3$ is contained in a unique copy of $D_{p+1}$. Since $x$ and $y$ have the same solvabilizers covering the same elements of order $\frac{p+1}{2}$, $x$ and $y$ are contained in the same copy of $D_{p+1}$. On the other hand, $D_{p+1}$ has a unique cyclic subgroup isomorphic to $C_{(p+1)/2}$ and the other elements are all involutions. Thus $D_{p+1}$ has exactly two elements of order $3$ and so we have either $y=x$ or $y=x^2$, which implies that they are in the same maximal subgroups.
\\ \\
The last case we need to check is when there are some elements of order $4$. It is seen from Tables \ref{tbl2,p,1}-\ref{tbl2,p,23} that that all cases lie in the two cases where $p\equiv 1\pmod 8$ and
$p\equiv 7\pmod 8$. First let $p\equiv 1\pmod 8$. Then any element of order $4$ is contained in two copies of $C_p \rtimes C_{(p - 1)/2}$ , one copy of $D_{p - 1}$ and some copies of $S_4$. We can apply the same reasoning as mentioned above to show that the two elements $x$ and $y$ of order $4$ with the same solvabilizers are contained in the same copies of $C_p \rtimes C_{(p - 1)/2}$. By taking $z$ as an element of order $\frac{p-1}{2}$ such that  $x \in \langle z \rangle$, we can find that $z$ is in $H_1$ and $H_2$ and so $H_1 \cap H_2 = \langle z \rangle$. Therefore, we have $y \in H_1 \cap H_2 = \langle z \rangle$.  Since $\langle z \rangle$ has exactly two elements of order $4$, we have either $y=x$ or $y=x^3$. Hence, they are in the same maximal subgroups. If $p\equiv 7\pmod 8$, then $x$ and $y$ are not in any copies of $C_p \rtimes C_{(p - 1)/2}$.  In this case, it is observed that any element of order $4$ is contained in a unique copy of $D_{p+1}$.
So they are contained in the same unique copy of $D_{p+1}$. Since $x$ and $y$ have the same solvabilizers covering the same elements of order $\frac{p+1}{2}$, $x$ and $y$ are contained in the same copy of $D_{p+1}$.  On the other hand, $D_{p+1}$ has a unique cyclic subgroup isomorphic to $C_{(p+1)/2}$ and the other elements are all involutions. Thus $D_{p+1}$ has exactly two elements of order $4$ and so we have either $y=x$ or $y=x^3$ which implies that they are in the same maximal subgroups.
\end{proof}

We can now determine  $|\Solv(G)|$ for all minimal simple groups $G$.
\begin{theorem}\label{solvpsl(2,2^p)}
    $|\Solv(\PSL(2, q))| = 2q^2$, where $q=2^p$ and $p$ is a prime.
\end{theorem}

\begin{proof}
   Let $G=\PSL(2, q)$.  We can proceed by counting the solvabilizers using Lemma \ref{maximalpsl(2,2^p)}.
     \\ \\
    In the case where $x$ is an involution, every involution has a distinct solvabilizer.  So the number of solvabilizers in this case is just the number of involutions, which is $q^2 - 1$.
    \\ \\
    In the case where $x$ and $y$ have order dividing $q - 1$, we see from the lemma that $x, y$ have the same solvabilizer if and only if they are contained in the same copy of $C_{q-1}$.  So the number of solvabilizers of such elements is the number of copies of $C_{q-1}$, which is $\frac{q(q+1)}{2}$ by \cites[Theorem 2.1]{King}.
    \\ \\
    In the case where $x$ has order dividing $q + 1$, an element $y$ has the same solvabilizer if and only if $x, y$ are contained in the same copy of $C_{q+1}$.  So the number of solvabilizers is the number of copies of $C_{q+1}$, which is $\frac{q(q-1)}{2}$ by \cites[Theorem 2.1]{King}.
    \\ \\
    In total, we obtain $|\Solv(G)| = (q^2 - 1) + \frac{q(q+1)}{2} + \frac{q(q-1)}{2} + 1 = 2q^2$, where we add $1$ to include the solvabilizer of the identity.
\end{proof}

\begin{theorem}
    $|\Solv(\PSL(2, q))| = \frac{4q^2 - q + 1}{2}$, where $q=3^p$ and $p$ is an odd prime.
\end{theorem}

\begin{proof}
    Let $G=\PSL(2, q)$. We apply Lemma \ref{maximalpsl(2,3^p)} to count the solvabilizers.
    \\ \\
    We see from the lemma that every two involutions have different solvabilizers. Considering \cites[Theorem 2.1]{King} and the fact that $3^p \equiv 3 \pmod 8$, we count $\frac{q(q-1)}{2}$ solvabilizers of involutions.
    \\ \\
    For the two elements $x$ and $y$ having order $3$, $x$ and $ y$ have the same solvabilizer if and only if they are in the same copy of $C_3$.  So the number of solvabilizers of elements of order $3$ is the number of copies of $C_3$, which is $\frac{q^2 - 1}{2}$.
    \\ \\
    In the case where $x$ and $y$ have order dividing $q - 1$, $x, y$ have the same solvabilizer if and only if they are contained in the same copy of $C_{(q-1)/2}$.  So the number of solvabilizers of such elements is the number of copies of $C_{(q-1)/2}$, which is $\frac{q(q+1)}{2}$ by \cites[Theorem 2.1]{King}.
    \\ \\
    In the case where $x$ and $y$ have order dividing $q + 1$, we see that $x$ and $y$ have the same solvabilizer if and only if they are contained in the same copy of $C_{(q+1)/2}$.  So the number of solvabilizers is the number of copies of $C_{(q+1)/2}$, which is $\frac{q(q-1)}{2}$ by \cites[Theorem 2.1]{King}.
    \\ \\
    In total, we obtain $|\Solv(G)| = \frac{q(q-1)}{2} + \frac{q^2 - 1}{2} + \frac{q(q+1)}{2} + \frac{q(q-1)}{2} + 1 = \frac{4q^2 - q + 1}{2}$, where we add $1$ to account for the solvabilizer of the identity.
\end{proof}

\begin{theorem}
    $|\Solv(\Sz(q))| = \frac{3q^4+q^3-q^2+q}{2}$, where $q=2^p$ and $p$ is an odd prime.
\end{theorem}
\begin{proof}
We start with counting distinct solvabilizers produced by involutions. It is seen from Lemma \ref{maximalSZ(2,2^p)} that the solvabilizers of any two involutions are distinct. So we need to count all involutions. According to \cite{BrayHolt}, $G$ has $q^2 + 1$ Sylow $2$-subgroups, resulting in $q^2 + 1$ maximal subgroups of the form $(C_2^p .C_2^p) \rtimes C_{q-1}$. Therefore, the number of involutions in total is $(q^2 + 1)(q - 1)$.
\\ \\
Next count distinct solvabilizers of elements of order $4$ by applying Lemma \ref{maximalSZ(2,2^p)}. Note that any two elements $x$ and $y$ of order $4$ having the same solvabilizer, are contained in the same maximal subgroups and so the same copy of $C_4$ and so $y = x$ or $y= x^3$. Thus any element of order $4$ has the same solvabilizer as exactly one other element of order $4$. As a result of \cite{SuzukiType}, there are $q(q^2+1)(q-1)$ elements of order $4$ which implies that there are $\frac{q(q^2+1)(q-1)}{2}$ different solvabilizers for elements of order $4$.
\\ \\
We next count distinct solvabilizers of elements of order dividing $q-1$. We proved in Lemma \ref{maximalSZ(2,2^p)} that given two elements of order dividing $q-1$,  they have the same solvabilizer if and only if they are contained in the same maximal subgroups, which holds if and only if they are in the same copy of $C_{q-1}$. So we just need to count the number of copies of $C_{q-1}$ in $G$. It is found in \cite{SuzukiType} that the number of elements of order $i$ as a divisor of $q- 1$ is $\varphi(i)q^2(q^2 + 1)/2$ where $\varphi$ is Euler's Totient Function. This means that there are exactly $\frac{q^2(q^2+1)}{2}$ copies of $C_{q-1}$. Therefore, the number of distinct solvabilizers of such elements in $G$ is $\frac{q^2(q^2+1)}{2}$.
\\ \\
Now compute the number of distinct solvabilizers produced by the elements of order dividing $q_+$. As the solvabilizer of such element is a copy of $C_{q_+} \rtimes C_4$, we need to count the number of copies of $C_{q_+} \rtimes C_4$ to find all distinct solvabilizers. 
It has been stated in \cite{BrayHolt} that up to conjugacy there is a maximal subgroup of type $C_{q_+} \rtimes C_4$ which is the normalizer of cyclic groups of orders $q_-$. So it is evident that the number of maximal subgroups $C_{q_+} \rtimes C_4$ is the following:
$$|G:N_G(C_{q_-})|=|G:C_{q_+} \rtimes C_4|=\frac{ q^2(q-1)(q^2+1)}{4q_+}=\frac{q^2(q_-)(q-1)}{4}.$$
\\
By a similar argument, we obtain that the number of distinct solvabilizers of elements of order dividing $q_-$ is $\frac{q^2(q_+)(q-1)}{4}$.
\\
Finally, we can find the total number of distinct solvabilizers in $G$ by adding up $$(q^2 + 1)(q - 1) + \frac{q(q^2+1)(q-1)}{2} + \frac{q^2(q^2+1)}{2} + \frac{q^2(q_-)(q-1)}{4} + \frac{q^2(q_+)(q-1)}{4}$$
$$ + 1 = \frac{3q^4 + q^3 - q^2 + q}{2},$$ where we add $1$ to account for the solvabilizer of the identity.
\end{proof}

\begin{theorem}\label{psl(2,p)}
$|\Solv(\PSL(2, p))|$, where $p>7$ is a prime number such that $p \equiv 2$ or $3 \pmod 5$ has the following values: 
    \begin{itemize}
        \item[{\rm (1)}] $\frac{5}{2}p^2 + \frac{5}{2}p + 2$ when $p\equiv 1\pmod{24}$,
        \item[{\rm (2)}] $2p^2 + p + 2$ when $p\equiv 5\pmod{24}$,
        \item[{\rm (3)}] $\frac{5}{2}p^2 + \frac{1}{2}p + 2$ when $p\equiv 7\pmod{24}$,
        \item[{\rm (4)}] $2p^2 + 2$ when $p\equiv 11\pmod{24}$,
        \item[{\rm (5)}] $2p^2 + 2p + 2$ when $p\equiv 13\pmod{24}$, 
        \item[{\rm (6)}] $\frac{5}{2}p^2 + \frac{3}{2}p + 2$ when $p\equiv 17\pmod{24}$,
        \item[{\rm (7)}] $2p^2 + p + 2$ when $p\equiv 19\pmod{24}$,
        \item[{\rm (8)}] $\frac{5}{2}p^2 - \frac{1}{2}p + 2$ when $p\equiv 23\pmod{24}$. 
    \end{itemize}
\end{theorem}
\begin{proof}
We exclude $p = 7$ as a special case because then $\frac{p-1}{2}=3$ and $\frac{p+1}{2}4$, so we overcount the solvabilizers of elements of order $3$ and $4$.  Using GAP, We found that $|\text{Solv}(\PSL(2, 7))| = 79$ in this case.  However, when $p \geq 11$, $\frac{p \pm 1}{2} \geq 5$ and so we do not have this issue, and our calculation is correct.
\\ \\
As a result of Lemma \ref{maximalPSL(2,p)}, the solvabilizers of involutions are distinct from each other. So the number of solvabilizers of involutions is the same as the number of involutions which is $\frac{p(p \pm 1)}{2}$ depending on what $p$ is modulo $4$.
\\ \\
To count the number of solvabilizers of elements of order $3$, we need to count the number of copies of $C_3$ because we showed in Lemma \ref{maximalPSL(2,p)} that if two elements $x$ and $y$ of order $3$ have the same solvabilizer, then $y=x$ or $y=x^2$.  This number is $\frac{p(p \pm 1)}{2}$ depending on what $p$ is modulo $3$.
\\ \\
In the case where $p \equiv 1$ or $7$ $\pmod 8$, the number of solvabilizers of elements of order $4$ is exactly the number of copies of $C_4$, since we proved that if $x$ and $y$ of order $4$ have the same solvabilizer, then $y=x$ or $y=x^3$.  This number is $\frac{p(p \pm 1)}{2}$, depending on what $p$ is modulo $8$.
\\ \\
The number of solvabilizers of elements of order $p$ is the number of copies of $C_p \rtimes C_{(p - 1)/2}$, which is $p + 1$.
\\ \\
To count the number of solvabilizers of elements of order dividing $p - 1$, we compute the number of copies of $C_{(p-1)/2}$ because if $x$ and $y$ have the same solvabilizers, then they are contained in the same maximal subgroups, and the intersection of these maximal subgroups is a $C_{(p-1)/2}$.  The number of all copies of $C_{(p-1)/2}$ is $\frac{p(p+1)}{2}$.
\\ \\
At the end, the number of solvabilizers of elements of order dividing $p+1$ is the number of copies of $D_{p+1}$ which is $\frac{p(p-1)}{2}$.
\\ \\
Adding all of these together gives the desired numbers stated in the theorem.
\end{proof}

We can use GAP to compute the number of solvabilizers in $\PSL(3, 3)$ as the last group in the list of minimal simple groups. The code in GAP can be found at the link in section \ref{GAPCalculation}.
\begin{remark}\label{psl(3,3)}
    $|\Solv(\PSL(3, 3))| = 1562$.
\end{remark}

\begin{corollary}
  Let $G$ and $H$ be two minimal simple groups with  $|\Solv(G)|=|\Solv(H)|$. Then $G\cong H$.
\end{corollary}
\begin{proof}
 Considering Theorems \ref{solvpsl(2,2^p)} - \ref{psl(2,p)} and Remark \ref{psl(3,3)}, we can get the result by an easy computation.
\end{proof}

\section{$|\Solv(G)|$ for some other simple groups}\label{othersimplegroups}
We extend some results in section \ref{minimalsimplegroups} to the corresponding cases where $G$ is not a minimal simple group. In fact, we find $|\Solv(G)|$ for the projective special linear groups $\PSL(2, 2^n)$ with $n \geq 2$, $\PSL(2, 3^n)$ with $n$ as an odd integer, and
$\PSL(2, p)$ for any prime number $p>7$. We also demonstrate in this section that the tables in the Appendix can be generalized to provide the whole structure of all solvabilizers of elements in the corresponding non-minimal simple groups.
\begin{theorem}\label{solvpsl(2,2^n)}
    $|\Solv(\PSL(2, q))| = 2q^2$, where $q = 2^n$ and $n \geq 2$ is an integer.
\end{theorem}
To prove Theorem \ref{solvpsl(2,2^n)}, we first get the following result.
    \begin{lemma}\label{maxpsl(2,2^n)}
     Let $G=\PSL(2, q)$ where $q = 2^n$ and $n \geq 2$ is an integer. Then given element $x$ in $G$,  $\Sol_G(x)$ is the union of the copies of $C_2^n \rtimes C_{q-1}$, $D_{2(q-1)}$, or $D_{2(q+1)}$ containing $x$.  Specifically, every solvable subgroup is contained in some copies of $C_2^n \rtimes C_{q-1}$, $D_{2(q-1)}$, and $D_{2(q+1)}$.
    \end{lemma}
    \begin{proof}
        We show this by induction on $n$.  This is true when $n$ is a prime.  When $n$ is not a prime, then $\Sol_G(x) = \bigcup_{x\in H} H$ where $H$ is a solvable subgroup of $G$ containing $x$.  We will show that every $H$ is contained in some copies of $C_2^n \rtimes C_{q-1}$, $D_{2(q-1)}$, or $D_{2(q+1)}$. To see this, by \cite[Corollary 2.2]{King} any such $H$ must be contained in a maximal subgroup isomorphic to $C_2^n \rtimes C_{q-1}, D_{2(q-1)}, D_{2(q+1)}$, or $\PSL(2, q_0)$ where $q$ is a prime power of $q_0$.  Suppose $H \leq \PSL(2, q_0)$, where $q_0 = 2^{n_0}$ for some integer $n_0$.  Then by the inductive hypothesis, $H$ is contained in some copy of $C_2^{n_0} \rtimes C_{q_0-1}$, $D_{2(q_0-1)}$, or $D_{2(q_0+1)}$.
        \\ \\
        We will show that every copy of $D_{2(q_0 \pm 1)}$ is contained in a copy of $D_{2(q \pm 1)}$. Since $q$ is a prime power of $q_0$, $q_0 \pm 1$ divides $q + 1$ or $q - 1$, respectively.  So let $q_0 \pm 1 = d$, where we assume without loss of generality that $d$ divides $q - 1$ because we can employ a similar argument to the $q + 1$ case.  Then it is seen from \cite[Theorem 2.1]{King} that there are $\frac{q(q^2 - 1)}{2d}$ copies of $D_{2d}$, and $\frac{q(q^2 - 1)}{2(q - 1)}$ copies of $D_{2(q - 1)}$.  Each copy of $D_{2(q - 1)}$ contains $\frac{q - 1}{d}$ copies of $D_{2d}$, and each copy of $D_{2d}$ is contained in at most one copy of $D_{2(q - 1)}$ since otherwise the normalizer of the $C_d$ will be a subgroup which strictly contains the $D_{2(q-1)}$.  As $D_{2(q-1)}$ is a maximal subgroup of $G$, we obtain that $C_d$ is normal which is a contradiction.  So every copy of $D_{2d}$ is contained in exactly one copy of $D_{2(q \pm 1)}$, as desired.
        \\ \\ 
        Finally, every copy of $C_2^{n_0} \rtimes C_{q_0-1}$ is contained in a copy of $C_2^n \rtimes C_{q - 1}$ by \cite[Theorem 2.1 (m)]{King} which completes the proof.
    \end{proof}
    
    \begin{proof}[Proof of Theorem \ref{solvpsl(2,2^n)}]
           In view of Lemma \ref{maxpsl(2,2^n)}, we can form the same table as Table \ref{tbl2,2p}. in the Appendix for the solvabilizers of elements in $\PSL(2, 2^n)$. So our calculation is similar and we once again obtain $2q^2$.        
    \end{proof}

\begin{theorem}\label{solvpsl(2,3^n)}
$|\Solv(\PSL(2, q))| = \frac{4q^2-q+1}{2}$, where $q=3^n$ for an odd integer $n$.
\end{theorem}
First, the following result should be proven.
\begin{lemma}\label{maxspl(2,3^n)}
        Let $G=\PSL(2, q)$, where $q=3^n$ for an odd integer $n$. Then given element $x$ in $G$, $\Sol_G(x)$ is the union of the copies of $C_3^n \rtimes C_{(q-1)/2}$, $D_{(q-1)}$, $D_{(q+1)}$,  and $A_4$ containing $x$.  Specifically, every solvable subgroup is contained in some copies of $C_3^n \rtimes C_{(q-1)/2}$, $D_{(q-1)}$, $D_{(q+1)}$, or $A_4$.
    \end{lemma}
    \begin{proof}
        We show this by induction on $n$.  This is true when $n$ is a prime.  When $n$ is not a prime, then $\Sol_G(x) = \bigcup_{x\in H} H$ where $H$ is a solvable subgroup of $G$ containing $x$.  We will show that every $H$ is contained in some copies of $C_3^n \rtimes C_{(q-1)/2}$, $D_{(q-1)}$, $D_{(q+1)}$, or $A_4$.  To see this, by \cite[Corollary 2.2]{King} any such $H$ must be contained in a maximal subgroup isomorphic to $C_3^n \rtimes C_{(q-1)/2}, D_{(q-1)}, D_{(q+1)}, A_4$, or $\PSL(2, q_0)$ where $q$ is a prime power of $q_0$.  Suppose $H \leq \PSL(2, q_0)$, where $q_0 = 3^{n_0}$ for some integer $n_0$.  Then by the inductive hypothesis, $H$ is contained in some copy of $C_3^{n_0} \rtimes C_{(q_{0}-1)/2}$, $D_{(q_0-1)}$, $D_{(q_0+1)}$, or $A_4$.
        \\ \\
        By a similar argument to Lemma \ref{maxpsl(2,2^n)}, we can see that every copy of $D_{(q_0 \pm 1)}$ is contained in a copy of $D_{(q \pm 1)}$.  Then by \cite[Theorem 2.1]{King} there are $\frac{q(q^2 - 1)}{2d}$ copies of $D_{d}$, and $\frac{q(q^2 - 1)}{2(q - 1)}$ copies of $D_{(q - 1)}$.  Each copy of $D_{(q - 1)}$ contains $\frac{q - 1}{d}$ copies of $D_{d}$, and each copy of $D_{2d}$ is contained in at most one copy of $D_{(q - 1)}$ since otherwise the normalizer of the $C_{d/2}$ will be a subgroup which strictly contains the $D_{q-1}$.  Since $D_{q-1}$ is maximal in $G$, we obtain that $C_{d/2}$ is normal, a contradiction.  So every copy of $D_{d}$ is contained in exactly one copy of $D_{(q \pm 1)}$, as desired.
        \\ \\ 
        Finally, every copy of $C_3^{n_0} \rtimes C_{(q_{0}-1)/2}$ is contained in a copy of $C_3^n \rtimes C_{(q-1)/2}$ by \cite[Theorem 2.1 (m)]{King}. So the lemma is proven.
    \end{proof}

   \begin{proof}[Proof of Theorem \ref{solvpsl(2,3^n)}]
       By applying Lemma \ref{maxspl(2,3^n)}, we can tabulate the same solvabilizer structures for $\PSL(2, 3^n)$ as Table \ref{tbl2,3p} in the Appendix for $\PSL(2, 3^p)$. Therefore, a similar computation implies a result as desired.
   \end{proof}

\begin{theorem}\label{solvpsl(2,p)}
Let $G=\PSL(2,p)$ where $p>7$ is a prime. Then $|\Solv(\PSL(2, p))|$ has the following values:
    \begin{itemize}
        \item[{\rm (1)}] $\frac{5}{2}p^2 + \frac{5}{2}p + 2$ when $p\equiv 1\pmod{24}$,
        \item[{\rm (2)}] $2p^2 + p + 2$ when $p\equiv 5\pmod{24}$,
        \item[{\rm (3)}] $\frac{5}{2}p^2 + \frac{1}{2}p + 2$ when $p\equiv 7\pmod{24}$,
        \item[{\rm (4)}] $2p^2 + 2$ when $p\equiv 11\pmod{24}$,
        \item[{\rm (5)}] $2p^2 + 2p + 2$ when $p\equiv 13\pmod{24}$, 
        \item[{\rm (6)}] $\frac{5}{2}p^2 + \frac{3}{2}p + 2$ when $p\equiv 17\pmod{24}$,
        \item[{\rm (7)}] $2p^2 + p + 2$ when $p\equiv 19\pmod{24}$,
        \item[{\rm (8)}] $\frac{5}{2}p^2 - \frac{1}{2}p + 2$ when $p\equiv 23\pmod{24}$. 
    \end{itemize}
(Note that we now have these formulas for all primes, not just when $p \equiv 2, 3 \pmod 5$.)
\end{theorem}
    \begin{lemma}\label{maxpsl(2,p)}
      Let $G=\PSL(2,p)$ where $p>7$ is a prime. For an element $x$ in $G$,  $\Sol_G(x)$ is the union of the copies of $C_p \rtimes C_{\frac{p-1}{2}}$, $A_4$ or $S_4$ (depending on $p$ mod $8$), and $D_{p\pm 1}$ containing $x$.  Specifically, every solvable subgroup is contained in a copy of $C_p \rtimes C_{\frac{p-1}{2}}$, $A_4$ or $S_4$ (depending on $p$ mod $8$), or $D_{p\pm 1}$.
    \end{lemma}
    \begin{proof}
        By \cite[Corollary 2.2]{King}, every solvable subgroup must be contained in some $C_p \rtimes C_{\frac{p-1}{2}}$, $A_4$ or $S_4$ (depending on $p$ mod $8$), $D_{p\pm 1}$, or $A_5$ (in the case where $p \equiv \pm 1 \pmod {10}$).  If $p \not\equiv \pm 1 \pmod {10}$, there are no copies of $A_5$, so we are done.  
        \\ \\
        Therefore, we focus on the case where $p \equiv \pm 1 \pmod {10}$.
        \\ \\
        Specifically, we need to show that every solvable subgroup of a copy of $A_5$ is contained in some $C_p \rtimes C_{\frac{p-1}{2}}$, $A_4$ or $S_4$ (depending on $p$ mod $8$), $D_{p\pm 1}$, or $A_5$ (in the case that $p \equiv \pm 1 \pmod {10}$).  Note that any subgroup of a copy of $A_5$ whose order is not divisible by $5$ is contained in a copy of $S_3$, $A_4$, or $D_{10}$ since these are the maximal subgroups of $A_5$.  
        \\ \\
        For the $S_3$ case, since $S_3 \cong D_6$ and $3$ divides one of $p \pm 1$, by the same argument as  Lemma \ref{maxspl(2,3^n)}, every copy of $D_6$ is contained in a copy of $D_{p \pm 1}$ .
        \\ \\
        The $D_{10}$ case also follows by the same argument as Lemma \ref{maxspl(2,3^n)}, since $5$ divides either $p-1$ or $p+1$ when $p \equiv \pm 1 \pmod {10}$.
        \\ \\
        For the $A_4$ case, if $p \equiv \pm 3 \pmod 8$, then $A_4$ is maximal.  Otherwise, there are $\frac{q(q^2 - 1)}{24}$ copies of $S_4$ and the same number of copies of $A_4$.  Each copy of $S_4$ contains exactly one $A_4$ and each $A_4$ is in at most one copy of $S_4$ (since otherwise its normalizer will strictly contain $S_4$, but by maximality of $S_4$ this $A_4$ will be a normal subgroup, a contradiction). So each $A_4$ is contained in exactly one $S_4$. 
    \end{proof}

   \begin{proof}[Proof of Theorem \ref{solvpsl(2,p)}]
       We can see from Lemma \ref{maxpsl(2,p)} that the table for $\PSL(2, p)$ in the Appendix still holds even when $p \equiv \pm 1 \pmod  4$.
   \end{proof}

\section{$|\Solv(G)|$ for nonsolvable groups}\label{nonsolvablesection}
Using our results for the minimal simple groups, we can obtain a lower bound on $|\Solv(G)|$ for all nonsolvable groups.  We begin with a result for the non-abelian simple groups.
\begin{lemma}\label{simple32}
        For any non-abelian simple group $S$, $|\Solv(S)| \geq 32$.
    \end{lemma}
    \begin{proof}
        Considering Theorem \ref{groupcontainingminimalsimplegroups}, $S$ contains a subgroup $K$ which is a minimal simple group.  It is also seen from Lemma \ref{sovH<SolvG} that  $|\Solv(S)| \geq |\Solv(K)|$.  Using the results on $|\Solv(K)|$ for the minimal simple groups $K$ found in Section \ref{minimalsimplegroups}, we can observe that $|\Solv(K)| \geq 32$ giving the desired bound.
    \end{proof}

\begin{lemma}\label{compositionseriesG}
    For any nonsolvable group $G$, there is a non-abelian simple group $H$ with $|\Solv(G)| \geq |\Solv(H)|$. 
\end{lemma}
\begin{proof}
Let $1 \unlhd N_1  \unlhd\cdots \unlhd N_n = G$ be the composition series of $G$ with the composition factors $H_1 , H_2 , \cdots, H_n$. Suppose that $H_i = N_i/N_{i-1}$ is the smallest nonsolvable composition factor of $G$. Since for any $j<i$, $H_j$ is solvable, for all $j<i$, $N_j$ is solvable. Then by Lemma \ref{sovH<SolvG} and Lemma \ref{SolvG/R(G)}, we have $$|\Solv(H_i)| = |\Solv(N_i)| \leq |\Solv(G)|,$$ which completes the proof.
\end{proof}

We now use Lemma \ref{compositionseriesG} to establish the following lower bound for all nonsolvable groups $G$.

\begin{theorem}
    For all nonsolvable groups $G$, we have $|\Solv(G)| \geq 32$.
\end{theorem}
\begin{proof}
    By Lemma \ref{simple32} and Lemma \ref{compositionseriesG}, we have $|\Solv(G)| \geq |\Solv(H)| \geq 32$ for some nonsolvable composition factor $H$ of $G$. 
\end{proof}
We also suggest the following conjecture.
\begin{conjecture}
    If $|\Solv(G)| = 32$, then there exists a composition factor of $G$ isomorphic to $A_5$.
\end{conjecture}

We finally present a relation between $|\Solv(G)|$ and the solvabilizer number of a group $G$ which was introduced in \cite{SolvNum} as the following: A solvabilizer covering of a group $G$ is a subset $X \subseteq G$ such that $G = \bigcup_{x\in X} \Sol_G(x)$. The solvabilizer number of $G$ denoted by $\alpha(G)$ is the cardinality of a minimal solvabilizer covering of $G$. 
\\ \\
The solvabilizer number of groups has been extensively studied in \cite{SolvNum}, where $\alpha(G)$ is computed for some finite groups. 
\begin{remark}\label{alphaandsol}
    Given nonsolvable group $G$, $\alpha(G)\leq |\Solv(G)|-1$.
\end{remark}

Remark \ref{alphaandsol} provides an upper bound for $\alpha(G)$, when $|\Solv(G)|$ is known.  In particular, we can obtain concrete upper bounds for $\alpha(G)$ for all minimal simple groups.  On the other hand, when a lower bound on $\alpha(G)$ is known this also gives a lower bound on $|\Solv(G)|$.

\section{Routine for Calculating $|\Solv(G)|$ in GAP}\label{GAPCalculation}

We calculate $|\Solv(G)|$ for various groups $G$ using code in \url{https://github.com/benjaminvakil/SizeofSolv} implementing an algorithm described below.
\\ \\
The benefit of this algorithm lies not only in its speed, but in its adaptability: the same strategies we describe here would still work if the problem is adapted to other subgroup-closed classes of finite groups such as abelian groups or nilpotent groups. Such questions could be new avenues of study.
\\ \\
Our computational results agree with the theoretical results we have in Sections \ref{minimalsimplegroups} and \ref{othersimplegroups}.  

\subsection{General Strategy}

One can easily describe a ``naïve'' algorithm that calculates $|\Solv(G)|$: for all $x\in G$ simply calculate $\Sol_{G}(x)$ by checking if $\langle x,y\rangle$ is solvable for every $y\in G$; then compare each two solvabilizers for equality. This method is very obviously inefficient as an effective algorithm should, one, minimize the number of elements we need to calculate the solvabilizer of, and thus, two, minimize the number of comparisons between solvabilizers we need to make. This brings us to an important characterization of $\Solv(G)$:

\begin{theorem}\label{formula for |Solv(G)|}
    Let $x_{1},\ldots,x_{n}\in G$ be representatives of the different rational classes where a rational class $R\subseteq G$ with representative $x$ is defined to be the union of all conjugacy classes $(x^{i})^{G}$ where $\gcd(|x|,i)=1$. Then,
    $$|\Solv(G)|=\sum_{x\in \chi}[G:N_{G}(\Sol_{G}(x))]$$
    where $\chi\subseteq \{x_{1},\ldots,x_{n}\}$ is the largest subset such that for any $x_{i},x_{j}\in \chi$, $\Sol_{G}(x_{i})\neq \Sol_{G}(x_{j})^{g}$ for all $g\in G$.

    \begin{proof}
      First note that given elements $x, y\in G$, $\Sol_{G}(x^{g})=\Sol_{G}(x)^{g}$ and $\Sol_{G}(x^{i})=\Sol_G(x)$ where $\gcd(|x|,i)=1$ which follow immediately from the fact that $\langle x, y\rangle^g = \langle x^g, y^g \rangle$ and $\langle x^i, y \rangle = \langle x, y \rangle$ for any $y \in G$.
      \\ \\
      Then, for any rational class $R$ with representative $x$, using the characterization that any $y\in R$  is of the form $(x^{i})^{g}$ where $\gcd(|x|,i)=1$ we have $$\{\Sol_{G}(y):y\in R\}=\{\Sol_{G}(x)^{g}:g\in G\}.$$  If $X\subseteq G$ is a subset (not necessarily a subgroup), then $X^{g_{1}}=X^{g_{2}}$ if and only if $g_{1},g_{2}\in G$ are in the same right coset of $N_{G}(X)$ and thus $|\{X^{g}:g\in G\}|=[G:N_{G}(X)]$. Specializing $X=\Sol_{G}(x)$, we get $$|\{\Sol_{G}(y):y\in R\}|=[G:N_{G}(\Sol_{G}(x))].$$

        Now, suppose $x,y\in G$ are in different rational classes and $\Sol_{G}(x)=\Sol_{G}(y)$. Say, $x\in R_{i}$ a rational class with representative $x_{i}$ and $y\in R_{j}$ a rational class with representative $x_{j}$. Then every other element of $\{\Sol_{G}(y):y\in R\}$ can be written as $\Sol_G(y)^g = \Sol_G(x)^g$ for some $g \in G$.  But $\Sol_G(x)^g$ is already counted in $[G:N_{G}(\Sol_{G}(x))]$, so restricting to the subset $\chi$ gives the correct count for $|\Solv(G)|$.
    \end{proof}

\end{theorem}

\subsection{Calculating $\Sol_{G}(x)$}

As mentioned in Section \ref{minimalsimplegroups}, for an element $x\in G$, $\Sol_{G}(x)$ as a union of all maximal solvable subgroups containing $x$. In general, there will be a large number of maximal subgroups (and thus maximal solvable subgroups); in fact, \cites[Chapter 39.19-8]{GAP} warns against calling the GAP command \texttt{MaximalSubgroups}, due to the amount of space it takes. Instead, we want to compute conjugacy class representatives of maximal solvable subgroups, from which we can access all maximal solvable subgroups. GAP does not have a command that does this natively, but we can write a command, \texttt{MaxSolvSubgroupsReps}, recursively that does this by indexing through $\texttt{H}\in\texttt{MaximalSubgroupClassReps(G)}$ and saving \texttt{H} to a list if it is solvable and again calling $\texttt{MaxSolvSubgroupsReps(H)}$ if it is not. We also have to be careful to remove redundancies as that inefficiency can greatly slow down later steps. Specifically, we do not want to include both $H$ and $H^{g}$ or $H$ and $K$ such that $K^{g}\leqslant H$ in our list.
\\ \\
Once we have found the maximal solvable subgroup class representatives, we may find $\Sol_{G}(x)$ for certain $x$—in our case $x_{2},x_{3},\ldots,x_{n}$ the rational class representatives removed $x_{1}=e$ as $\Sol_{G}(e)=G$ is trivial. To do this, for each maximal solvable subgroup class representative $H$, we test for which $x\in H^{g}$ subgroups conjugate to $H$ which we deposit in the local variable \texttt{subgrpsinunion}; equivalently, we test for $x^{g^{-1}}\in H$ to avoid having to create the group $H^{g}$ if not necessary. Then, the algorithm goes (using GAP pseudo-code for illustration):

\vspace{.3cm}

\begin{tcolorbox}[colback=black!5!white,colframe=black,breakable,enhanced jigsaw]
\texttt{> if RemInt($|H|$, $|x|$)=0 then}\\
\texttt{> \hspace{.3in}for $g\in$ RightTransversal($G$, $N_{G}(H)$) do}\\
\texttt{> \hspace{.6in}if $x^{g^{-1}}\in H$ then}\\
\texttt{> \hspace{.9in}Add(subgrpsinunion, $H^{g}$);}\\
\texttt{> \hspace{.6in}fi;}\\
\texttt{> \hspace{.3in}od;}\\
\texttt{> fi;}
\end{tcolorbox}
\vspace{.3 cm}
\subsection{Calculating the Normalizer of a Solvabilizer}

\begin{theorem}\label{normalizer of a solvabilizer}
    $N_{G}(\langle x\rangle)\leqslant N_{G}(\Sol_{G}(x))$.
\end{theorem}

\begin{proof}
    Let $y\in \Sol_{G}(x)$ and $h\in N_{G}(\langle x\rangle)$. So, $\langle x,y\rangle\cong \langle x,y\rangle^{h}$ is solvable. Moreover, $\langle x,y\rangle^{h}=\langle x^{h},y^{h}\rangle=\langle x^{i},y^{h}\rangle$ for some $i$ coprime to $|x|$ since $h$ normalizes $\langle x\rangle$. As $\gcd(|x|,i)=1$, $y^{h}\in \Sol_{G}(x^{i})=\Sol_{G}(x)$ and thus $h$ normalizes $\Sol_{G}(x)$. Therefore, $N_{G}(\langle x\rangle)\leqslant N_{G}(\Sol_{G}(x))$, by taking this fact over all $h\in N_{G}(\langle x\rangle)$.
\end{proof}

\begin{corollary}
    Summing over all representatives $x_{1},\ldots,x_{n}$ of rational classes in $G$, we have an upper bound:
    $$|\Solv(G)|\leqslant \sum_{i=1}^{n}[G:N_{G}(\langle x_i\rangle)]$$
\end{corollary}
\begin{proof}
    This follows immediately from Theorem \ref{formula for |Solv(G)|} and Theorem \ref{normalizer of a solvabilizer}.
\end{proof}

In the algorithm, we use Theorem \ref{normalizer of a solvabilizer} to speed up calculating the normalizer of $\Sol_G(x)$.  Instead of iterating through all elements $g \in G$ and checking if $\Sol_G(x) = \Sol_G(x)^g$, it suffices to iterate through one representative for each coset of $N_G(x)$ in $H$.  This is because if $g_1$ and $g_2$  are in the same coset of $N_G(\langle x \rangle)$, we see from Theorem \ref{normalizer of a solvabilizer} that $\Sol_G(x)^{g_1} = \Sol_G(x)^{g_2}$.

\subsection{Computing $|\Solv(G)|$}
After speeding up the calculation of $N_G(\Sol_G(x))$, it remains to determine the set $\chi$ in Theorem \ref{formula for |Solv(G)|}, but the number of elements we must check is greatly diminished. Consider two representatives $x_i$ and $x_j$ of different rational classes, we determine whether or not two elements of the respective rational classes have the same solvabilizers. This relates to Question \ref{Sol(x)=Sol(y) implies same subgroups}, and this step could be mostely forgone if we \textit{a priori} knew whether our group $G$ satisfies that Question. Notably, even though we do have to index over pairs of elements in $G$ for this, which is time-consuming, the number of elements we now need to do so is much smaller than the naive algorithm, as there are generally far fewer rational classes than elements.
\\ \\
Now, this is equivalent to determining if $\Sol_G(x_i) = \Sol_G(x_j)^{g_0}$ for some $g_0 \in G$. To find such a $g_0$, we first check that $|\Sol_G(x_i)| = |\Sol_G(x_j)|$.  Then, note that if $\Sol_G(x_i) = \Sol_G(x_j)^{g_0}$, then $N_G(\Sol_G(x_i)) = N_G(\Sol_G(x_j))^{g_0}$, so we search for a $g \in G$ satisfying $N_G(\Sol_G(x_i)) = N_G(\Sol_G(x_j))^{g}$ using the built-in function \texttt{RepresentativeAction} in GAP.
\\ \\
Then $g$ and $g_0$ must be in the same coset of the normalizer of $N_G(\Sol_G(x_j))$.  We iterate over every element $h \in H = N_G(N_G(\Sol_G(x_j))$, and check if $\Sol_G(x_i) = \Sol_G(x_j)^{gh}$ by simply conjugating every element of the solvabilizer by $gh$.  We know that if $g_0$ exists, it must be in the same coset of $H$ as $g$.  So through this process, we are guaranteed find $g_0$ if it exists.  
\\ \\
Finally, in the case that $g_0$ exists, we then delete one of the $x_i$ or $x_j$ from our list of representatives of rational classes.  Repeating this process for all possible pairs of $x_i, x_j$ gives the desired set $\chi$, completing the algorithm.
\newpage
\section{Appendix: Tables from \cite{tables}}
\vspace{.5cm}
These are tables of data about the solvabilizer, which were proven in section 3 of \cite{tables}.
\vspace{1cm}

\begin{center}
\small\addtolength{\tabcolsep}{-1pt}
\begin{table}[h!]\begin{center}
    \begin{tabular}{ |c |  c|  c|  c|} 
     \hline
     \quad Maximal Subgroup & $|x| = 2$ & $|x| \mid q-1$ &  $|x| \mid q+1$   \\ [0.8ex] 
     \hline\hline
     $C_2^p \rtimes C_{q-1}$  & 1  & 2  & 0   \\ 
     \hline
     $D_{2(q-1)}$  & $q/2$ & 1 & 0  \\
     \hline
     $D_{2(q+1)}$  & $q/2$ & 0 & 1   \\
     \hline\hline
     $|\Sol_G(x)|$ & $3q(q-1)$ & $2q(q-1)$ & $2(q+1)$ \\
     \hline
    \end{tabular}\vspace{.5cm}
    \caption{Classification of $\Sol_{G}(x)$ for $G={\rm PSL}(2, 2^p)$}
    \label{tbl2,2p}
    \end{center}
\vspace{1cm}
\begin{center}
\begin{tabular}{| c | c | c | c | c|} 
 \hline
 \quad Maximal Subgroup &  $|x| = 2$ & $|x| = 3$ & $|x| \mid q-1$ & $|x| \mid q+1$  \\ [0.5ex] 
 \hline\hline
 $C_3^p \rtimes C_{(q-1)/2}$ & 0 & 1 & 2 & 0  \\ 
 \hline
 $D_{q-1}$ & $(q+1)/2$& 0 & 1 & 0  \\
 \hline
 $D_{q+1}$ & $(q+3)/2$ & 0 & 0 & 1  \\
 \hline
 $\rm{A}_4$ & $(q+1)/4$ & $q/3$ & 0 & 0 \\
 \hline \hline
 $|\Sol_G(x)|$ & $q(q+1)$ & $q(q+5)/2$ & $q(q-1)$ & $q+1$ \\
 \hline
\end{tabular}\vspace{.5cm}
    \caption{Classification of $\Sol_{G}(x)$ for  $G={\rm PSL}(2,3^p)$}
    \label{tbl2,3p}
    \end{center}
\vspace{1cm}
\begin{center}
\begin{tabular}{| c || c | c | c | c | c |} 
 \hline
 \quad Maximal Subgroup & $|x| = 2$ & $|x| = 4$ & $|x| \mid q-1$ & $|x| \mid q_+$ & $|x|\mid q_-$  \\ [0.5ex] 
 \hline\hline
 $C_2^p .C_2^p) \rtimes C_{q-1}$ & 1 & 1 & 2 & 0 & 0 \\ 
 \hline
 $D_{2(q-1)}$ & $q^2/2$ & 0 & 1 & 0 & 0 \\
 \hline
 $C_{q_+}\rtimes C_4$ & $q^2/4$ & $q/2$ & 0 & 1 & 0 \\
 \hline
 $C_{q_-}\rtimes C_4$ & $q^2/4$ & $q/2$ & 0 & 0 & 1 \\
 \hline 
 $|\Sol_G(x)|$ & $q^2(4q-3)$ & $q^2(q+3)$ & $2q^2(q-1)$ & $4q_+$ & $4q_-$\\
 \hline
\end{tabular}\vspace{.5cm}
\caption{Classification of $\Sol_{G}(x)$ for $G=\text{Sz}(2^p)$}
    \label{tblsz,2p}
    \end{center}
\end{table}
\end{center}

\begin{center}
\small\addtolength{\tabcolsep}{-2pt}
\begin{table}[h!]\begin{center}
\begin{tabular}{| c || c | c | c | c | c | c |} 
 \hline
 \quad Maximal Subgroup & $|x| = 2$ & $|x| = 3$ & $|x| = 4 $ & $|x| = p$ & $|x| \mid p-1$ & $|x|\mid p+1$  \\ [0.5ex] 
 \hline\hline
 $\cp$ & 2 & 2 & 2 & 1 & 2 & 0 \\ 
 \hline
 $D_{p-1}$ & $(p+1)/2$ & 1 & 1 & 0 & 1 & 0 \\
 \hline
 $D_{p+1}$ & $(p-1)/2$ & 0 & 0 & 0 & 0 & 1 \\
 \hline
 $\rm{S}_4$ & $3(p-1)/4$ & $(p-1)/3$ & $(p-1)/4$ & 0 & 0 & 0 \\
 \hline \hline
 $|\Sol_{G}(x)|$ & $(p-1)(2p+3)$ & $(p-1)(p+6)$ & $(p-1)(p+4)$ & $p(p-1)/2$ & $p(p-1)$& $p+1$\\
 \hline
\end{tabular}
\vspace{.25cm}
\caption{Classification of $\Sol_{G}(x)$ for $G=\PSL(2, p)$ when $p\equiv 1\pmod{24}$}
    \label{tbl2,p,1}
     \end{center}
\vspace{.75cm}
\begin{center}
\begin{tabular}{| c || c | c | c | c | c |} 
 \hline
 \quad Maximal Subgroup & $|x| = 2$ & $|x| = 3$ & $|x| = p$ & $|x| \mid p-1$ & $|x|\mid p+1$  \\ [.5ex] 
 \hline\hline
 $\cp$ & 2 & 0 & 1 & 2 & 0 \\ 
 \hline
 $D_{p-1}$ & $(p+1)/2$ & 0 & 0 & 1 & 0 \\
 \hline
 $D_{p+1}$ & $(p-1)/2$ & 1 & 0 & 0 & 1 \\
 \hline
 $\rm{A}_4$ & $(p-1)/4$ & $(p+1)/3$ & 0 & 0 & 0 \\
 \hline \hline
 $|\Sol_{G}(x)|$ & $(p-1)(2p-1)$ & $4(p+1)$ & $p(p-1)/2$ & $p(p-1)$& $p+1$\\
 \hline
\end{tabular}
\vspace{.25cm}
\caption{Classification of $\Sol_{G}(x)$ for $G=\PSL(2, p)$ when $p\equiv 5\pmod{24}$}
    \label{tbl2,p,5}
    \end{center}
\vspace{.75cm}
\begin{center}
\begin{tabular}{| c || c | c | c | c | c | c |} 
 \hline
 \quad Maximal Subgroup & $|x| = 2$ & $|x| = 3$ & $|x| = 4 $ & $|x| = p$ & $|x| \mid p-1$ & $|x|\mid p+1$  \\ [0.5ex] 
 \hline\hline
 $\cp$ & 0 & 2 & 0 & 1 & 2 & 0 \\ 
 \hline
 $D_{p-1}$ & $(p+1)/2$ & 1 & 0 & 0 & 1 & 0 \\
 \hline
 $D_{p+1}$ & $(p+3)/2$ & 0 & 1 & 0 & 0 & 1 \\
 \hline
 $\rm{S}_4$ & $3(p+1)/4$ & $(p-1)/3$ & $(p+1)/4$ & 0 & 0 & 0 \\
 \hline \hline
 $|\Sol_{G}(x)|$ & $(p+1)(p+4)$ & $(p-1)(p+6)$ & $5(p+1)$ & $p(p-1)/2$ & $p(p-1)$& $p+1$\\
 \hline
\end{tabular}\vspace{.25cm}
\caption{Classification of $\Sol_{G}(x)$ for $G=\PSL(2, p)$ when $p\equiv 7\pmod{24}$}
    \label{tbl2,p,7}
\end{center}
\vspace{.75cm}
\begin{center}
\begin{tabular}{| c || c | c | c | c | c |} 
 \hline
 \quad Maximal Subgroup & $|x| = 2$ & $|x| = 3$ & $|x| = p$ & $|x| \mid p-1$ & $|x|\mid p+1$  \\ [0.5ex] 
 \hline\hline
 $\cp$ & 0 & 0 & 1 & 2 & 0 \\ 
 \hline
 $D_{p-1}$ & $(p+1)/2$ & 0 & 0 & 1 & 0 \\
 \hline
 $D_{p+1}$ & $(p+3)/2$ & 1 & 0 & 0 & 1 \\
 \hline
 $\rm{A}_4$ & $(p+1)/4$ & $(p+1)/3$ & 0 & 0 & 0 \\
 \hline \hline
 $|\Sol_{G}(x)|$ & $p(p+1)$ & $4(p+1)$ & $p(p-1)/2$ & $p(p-1)$& $p+1$\\
 \hline
\end{tabular}\vspace{.25cm}
\caption{Classification of $\Sol_{G}(x)$ for $G=\PSL(2, p)$ when $p\equiv 11\pmod{24}$}
    \label{tbl2,p,11}
\end{center}

\vspace{.75cm}
\end{table}
\begin{table}[h!]\begin{center}
\begin{tabular}{| c || c | c | c | c | c |} 
 \hline
 \quad Maximal Subgroup & $|x| = 2$ & $|x| = 3$ & $|x| = p$ & $|x| \mid p-1$ & $|x|\mid p+1$  \\ [0.5ex] 
 \hline\hline
 $\cp$ & 2 & 2 & 1 & 2 & 0 \\ 
 \hline
 $D_{p-1}$ & $(p+1)/2$ & 1 & 0 & 1 & 0 \\
 \hline
 $D_{p+1}$ & $(p-1)/2$ & 0 & 0 & 0 & 1 \\
 \hline
 $\rm{A}_4$ & $(p-1)/4$ & $(p-1)/3$ & 0 & 0 & 0 \\
 \hline \hline
 $|\Sol_{G}(x)|$ & $(p-1)(2p-1)$ & $(p-1)(p+3)$ & $p(p-1)/2$ & $p(p-1)$& $p+1$\\
 \hline
\end{tabular}
\vspace{.25cm}
\caption{Classification of $\Sol_{G}(x)$ for $G=\PSL(2, p)$ when $p\equiv 13\pmod{24}$}
    \label{tbl2,p,13}
    \end{center}
\vspace{.75cm}

\begin{tabular}{| c || c | c | c | c | c | c |} 
 \hline
 \quad Maximal Subgroup & $|x| = 2$ & $|x| = 3$ & $|x| = 4 $ & $|x| = p$ & $|x| \mid p-1$ & $|x|\mid p+1$  \\ [0.5ex] 
 \hline\hline
 $\cp$ & 2 & 0 & 2 & 1 & 2 & 0 \\ 
 \hline
 $D_{p-1}$ & $(p+1)/2$ & 0 & 1 & 0 & 1 & 0 \\
 \hline
 $D_{p+1}$ & $(p-1)/2$ & 1 & 0 & 0 & 0 & 1 \\
 \hline
 $\rm{S}_4$ & $3(p-1)/4$ & $(p+1)/3$ & $(p-1)/4$ & 0 & 0 & 0 \\
 \hline \hline
 $|\Sol_{G}(x)|$ & $(p-1)(2p+3)$ & $7(p+1)$ & $(p-1)(p+4)$ & $p(p-1)/2$ & $p(p-1)$& $p+1$\\
 \hline
\end{tabular}\vspace{.25cm}
\caption{Classification of $\Sol_{G}(x)$ for $G=\PSL(2, p)$ when $p\equiv 17\pmod{24}$}
    \label{tbl2,p,17}
\vspace{.75cm}
\begin{center}
\begin{tabular}{| c || c | c | c | c | c |} 
 \hline
 \quad Maximal Subgroup & $|x| = 2$ & $|x| = 3$ & $|x| = p$ & $|x| \mid p-1$ & $|x|\mid p+1$  \\ [0.5ex] 
 \hline\hline
 $\cp$ & 0 & 2 & 1 & 2 & 0 \\ 
 \hline
 $D_{p-1}$ & $(p+1)/2$ & 1 & 0 & 1 & 0 \\
 \hline
 $D_{p+1}$ & $(p+3)/2$ & 0 & 0 & 0 & 1 \\
 \hline
 $\rm{A}_4$ & $(p+1)/4$ & $(p-1)/3$ & 0 & 0 & 0 \\
 \hline \hline
 $|\Sol_{G}(x)|$ & $p(p+1)$ & $(p-1)(p+3)$ & $p(p-1)/2$ & $p(p-1)$& $p+1$\\
 \hline
\end{tabular}
\vspace{.25cm}
\caption{Classification of $\Sol_{G}(x)$ for $G=\PSL(2, p)$ when $p\equiv 19\pmod{24}$}
    \label{tbl2,p,19}
    \end{center}
\vspace{.75cm}
\begin{center}
\begin{tabular}{| c || c | c | c | c | c | c |} 
 \hline
 \quad Maximal Subgroup & $|x| = 2$ & $|x| = 3$ & $|x| = 4 $ & $|x| = p$ & $|x| \mid p-1$ & $|x|\mid p+1$  \\ [0.5ex] 
 \hline\hline
 $\cp$ & 0 & 0 & 0 & 1 & 2 & 0 \\ 
 \hline
 $D_{p-1}$ & $(p+1)/2$ & 0 & 0 & 0 & 1 & 0 \\
 \hline
 $D_{p+1}$ & $(p+3)/2$ & 1 & 1 & 0 & 0 & 1 \\
 \hline
 $\rm{S}_4$ & $3(p+1)/4$ & $(p+1)/3$ & $(p+1)/4$ & 0 & 0 & 0 \\
 \hline \hline
 $|\Sol_{G}(x)|$ & $(p+1)(p+4)$ & $7(p+1)$ & $5(p+1)$ & $p(p-1)/2$ & $p(p-1)$& $p+1$\\
 \hline
\end{tabular}
\vspace{.25cm}
\caption{Classification of $\Sol_{G}(x)$ for $G=\PSL(2, p)$ when $p\equiv 23\pmod{24}$}
    \label{tbl2,p,23}
    \end{center}
\end{table}
\end{center}

\nocite{*}
\printbibliography

\end{document}